%% file: main.tex
\documentclass[11pt]{article}
\usepackage[margin=1in,bottom=1in]{geometry}
\usepackage[utf8]{inputenc}

\usepackage{libertine}
\usepackage{amsthm}
\usepackage{amsmath}

\usepackage{amssymb}
\usepackage{stmaryrd}
\usepackage{mathrsfs} 
\usepackage{mathtools}
\usepackage{thm-restate}
\usepackage[colorlinks = true,linkcolor = blue,urlcolor = blue, citecolor = blue]{hyperref}
\usepackage[capitalize, nameinlink]{cleveref}
\usepackage{url}
\usepackage{todonotes}
\usepackage{mathrsfs} 
\usepackage{soul}
\usepackage{relsize}
\usepackage{enumerate}
\usepackage[all,defaultlines=3]{nowidow}
\usepackage{multicol}
\usepackage{enumitem}
\usepackage{xspace}
\usepackage{nicefrac}
\usepackage{wasysym}

\usepackage{textpos}
\setlist[itemize]{topsep=4pt,itemsep=3pt,parsep=0pt} 
\setlist[enumerate]{topsep=4pt,itemsep=3pt,parsep=0pt} 

\crefname{claim}{Claim}{Claims}
\crefname{figure}{Figure}{Figures}

\newtheorem{theorem}{Theorem}[section]

\newtheorem{observation}[theorem]{Observation}
\newtheorem{lemma}[theorem]{Lemma}

\newtheorem{claim}[theorem]{Claim}

\theoremstyle{definition}
\newtheorem{definition}[theorem]{Definition}
\theoremstyle{plain}

\theoremstyle{definition}

\newtheorem*{example*}{Example}

\numberwithin{equation}{section}

\newenvironment{claimproof}[1][Proof of the claim.]{%
  \begin{proof}[#1]%
}{%
  \end{proof}%
}

\newcommand{\Qq}{\mathcal{Q}}

\newcommand{\Xx}{\mathcal{X}}
\newcommand{\Yy}{\mathcal{Y}}

\newcommand{\N}{\mathbb{N}}

\newcommand{\Gaif}{\mathsf{Gaif}}
\newcommand{\Ball}{\mathsf{Ball}}
\newcommand{\dist}{\mathrm{dist}}

\newcommand{\Ff}{\mathcal{F}}
\newcommand{\Ss}{\mathcal{S}}

\newcommand{\aff}[1]{\textcolor{black!60}{\small{#1}}}

\newcommand{\Tf}{\mathsf{T}}

\newcommand{\Balls}{\mathsf{Balls}}

\newcommand{\Cc}{\mathscr C}
\newcommand{\Dd}{\mathscr D}
\newcommand{\Ee}{\mathscr E}
\def\epsilon{\varepsilon}

\renewcommand{\emptyset}{\varnothing}

\renewcommand{\leq}{\leqslant}
\renewcommand{\geq}{\geqslant}

\renewcommand{\setminus}{-}

\usepackage{tikz}

\usetikzlibrary{backgrounds}
\tikzset{node/.style={draw, circle, fill = black, minimum size = 3pt, inner sep=0pt, line width=1pt}}
\tikzset{nodesubwall/.style={draw, circle, fill = blue, minimum size = 3.1pt, inner sep=0pt}}
\tikzset{corner/.style={draw=magenta, fill = red!20!white, minimum size = 6pt,inner sep=0pt,line width=1pt}}
\tikzset{central/.style={draw=orange, fill = red!20!white, minimum size = 6pt,inner sep=0pt,line width=1pt}}
\tikzset{edge/.style={draw=white!60!black,line width=1.5pt}}
\tikzset{edgesubwall/.style={draw=blue!60!white,line width=2pt}}

\tikzset{subnode/.style={draw, circle, fill = yellow!50!red!50!white, minimum size = 2pt, inner sep=0pt}}
\tikzset{edgethin/.style={draw=white!60!black,line width=1.3pt}}

\begin{document}

\newcommand{\funding}{Mi. P. was supported by the project BOBR that is funded from the European Research Council (ERC) under the European Union’s Horizon 2020 research and innovation programme with grant agreement No. 948057.
}

\title{Strong odd colorings in graph classes of bounded expansion\footnote{\funding}}
\date{}
 \author{
   Michał Pilipczuk \\
   \aff{University of Warsaw} \\
   \aff{michal.pilipczuk@mimuw.edu.pl}
 }
\maketitle

\input{abstract.tex}

 \begin{textblock}{20}(-1.75, 6.1)
 \includegraphics[width=40px]{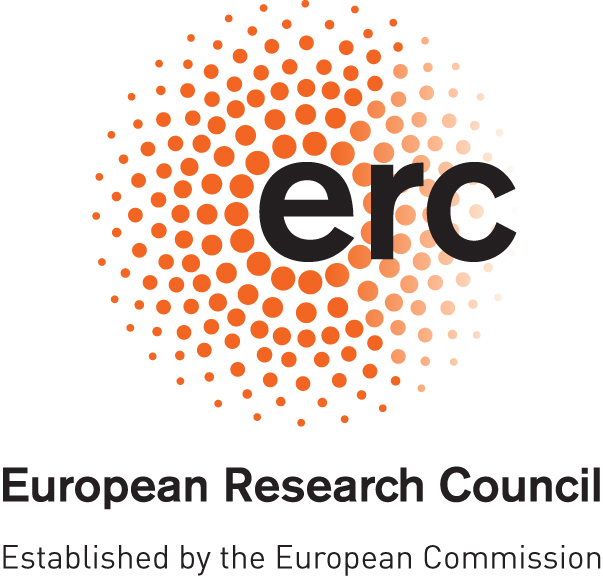}%
 \end{textblock}
 \begin{textblock}{20}(-1.75, 7.1)
 \includegraphics[width=40px]{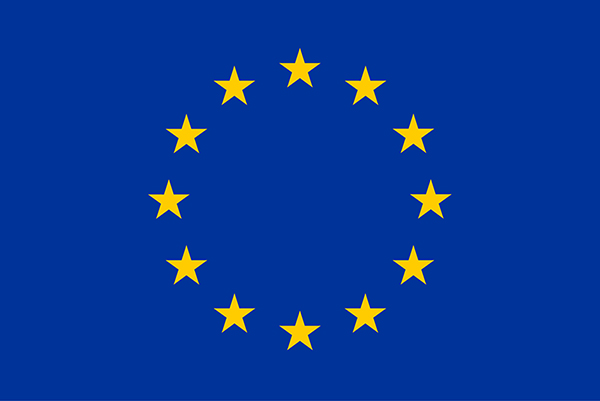}%
 \end{textblock}

\newpage

\clearpage
\setcounter{page}{1}

\input{intro}

\input{prelims}

\input{subchromatic}

\input{proof}

\input{conclusions}

\bibliographystyle{plain}
\bibliography{references}

\end{document}

%% file: abstract.tex
\begin{abstract}
	We prove that for every $d\in \N$ and a graph class of bounded expansion $\Cc$, there exists some $c\in \N$ so that every graph from $\Cc$ admits a proper coloring with at most $c$ colors satisfying the following condition: in every ball of radius~$d$, every color appears either zero times or an odd number of times. For $d=1$, this provides a positive answer to a question raised by Goetze, Klute, Knauer, Parada, Pe\~na, and Ueckerdt [ArXiv~2505.02736] about the boundedness of the strong odd chromatic number in graph classes of bounded expansion. The key technical ingredient towards the result is a proof that the strong odd coloring number of a sets system can be bounded in terms of its semi-ladder index, 2VC dimension, and the maximum subchromatic number among  induced subsystems.
\end{abstract}

%% file: intro.tex
\section{Introduction}\label{sec:intro}

Recall that a {\em{proper coloring}} of a graph $G$ is a function $\phi$ mapping vertices of $G$ to some palette of colors, say $C$, so that for every two vertices $u,v$ that are adjacent in $G$, we have $\phi(u)\neq \phi(v)$. Recently, Kwon and Park~\cite{KwonP24} introduced the concept of a {\em{strong odd coloring}} of a graph, which is a proper coloring that satisfies the following additional condition: for every vertex $u$ and every color $i\in C$, the number of vertices of color $i$ among the neighbors of $u$ is either zero or odd. This is a strengthening of the notion of an {\em{odd coloring}}, where we only require that the coloring is proper and for every vertex $u$ with a nonempty neighborhood, there exists a color that is featured an odd number of times among the neighbors of $u$. The {\em{(strong) odd chromatic number}} of a graph $G$ is defined as the minimum number of colors needed for a (strong) odd coloring of $G$.

Kwon and Park~\cite{KwonP24} proved several results relating the strong odd chromatic number of a graph and its maximum degree. They also asked whether the strong odd chromatic number is universally bounded on the class of planar graphs. This question was answered in positive by Caro et al.~\cite{CaroPST24}, who proved that the strong odd chromatic number of every planar graph is at most $388$. Very recently, Goetze et al.~\cite{GoetzeKKPPU25} generalized this result by proving that every proper minor-closed class has bounded strong odd chromatic number. They also asked whether this conclusion is true even in the larger generality of graph classes of bounded expansion, which is a concept of uniform, local sparsity in graphs that underlies the theory of Sparsity of Ne\v{s}et\v{r}il and Ossona de Mendez; see \cite{sparsity} for an introduction. Note that, as proved by Hickingbotham~\cite{Hickingbotham23} and by Liu~\cite{Liu24}, graph classes of bounded expansion have bounded odd chromatic~number.

\medskip

In this paper we answer the question of Goetze et al.~\cite{GoetzeKKPPU25} in positive.

\begin{theorem}\label{thm:main-be}
	Every graph class of bounded expansion has bounded strong odd chromatic number.
\end{theorem}

In fact, we prove a more general statement that speaks about balls of any fixed radius $d$, instead of just neighborhoods.

\begin{restatable}{theorem}{main}\label{thm:main}
	Let $d\in \N$ and $\Cc$ be a graph class of bounded expansion. Then there exists $c\in \N$ such that for every graph $G\in \Cc$ there exists a proper coloring of $G$ with $c$ colors satisfying the following condition: for every vertex $u$ of $G$ and color $i$, the number of vertices of color $i$ that are at distance at most $d$ from $u$ is either zero or odd.
\end{restatable}

Note that \cref{thm:main-be} follows from \cref{thm:main} for $d=1$. Indeed, if $\phi$ is a proper coloring of a graph $G$ such that in every radius-$1$ ball, say around a vertex $u$, every color is featured zero times or an odd number, then this is also true for the (open) neighborhood $N_G(u)$, because the color of $u$ is not featured among the colors present in $N_G(u)$ (due to $\phi$ being proper).

The core element of our proof is the following statement about strong odd colorings of set systems; that is, colorings of the universe of the set system so that every set in the family contains every color zero times or an odd number of times. Precisely, we bound the odd chromatic number of a set system in terms of three parameters: the semi-ladder index, the 2VC dimension, and the maximum subchromatic number among the induced subsystems (see \cref{sec:prelims,sec:subchromatic} for definitions).

\begin{restatable}{theorem}{setSystem}\label{thm:main-set-system}
	There is a function $f\colon \N^3\to \N$ such that the following holds. Suppose $\Ss$ is a set system with semi-ladder index at most $\ell$, 2VC dimension at most $h$, and whose every induced subsystem has subchromatic number at most $s$, for some $\ell,h,s\in \N$. Then the strong odd chromatic number of $\Ss$ is at most $f(\ell,h,s)$.
\end{restatable}

The proof of \cref{thm:main-set-system}, presented in \cref{sec:proof}, proceeds by induction on the semi-ladder index. The boundedness of the 2VC dimension is exploited through a duality result of Ding et al.~\cite{DingSW94}, whereas the boundedness of the subchromatic index is used to design the induction so that the semi-ladder index drops with every induction step.

To derive \cref{thm:main} from \cref{thm:main-set-system}, it suffices to show that for every $d\in \N$ and graph class of bounded expansion $\Cc$, the class of set systems of radius-$d$ balls in graphs from $\Cc$ has all the relevant parameters bounded. This was explicitly proved for the semi-ladder index by Fabia\'nski et al.~\cite{FabianskiPST19}, for the 2VC dimension by Pilipczuk and Siebertz~\cite{PilipczukS21}, and for the subchromatic number by Ne\v{s}et\v{r}il et al.~\cite{NesetrilMPZ20}. However, the last proof applies only to the whole set system of balls, and not to its induced subsystems. Hence, in \cref{sec:subchromatic} we propose a new argument that suitably strengthens the result of~\cite{NesetrilMPZ20}.

%% file: prelims.tex
\section{Preliminaries}\label{sec:prelims}

By $\N$ we denote the set of nonnegative integers. For a positive integer $p$, we write $[p]\coloneqq \{1,\ldots,p\}$.

\paragraph*{Graphs.} We use standard graph notation. All graphs in this paper are finite, undirected, and simple, that is, without loops or parallel edges. For a graph $G$, by $V(G)$ we denote the vertex set of $G$. For $A\subseteq V(G)$, $G[A]$ denotes the subgraph induced by $A$, which is the graph on vertex set $A$ where any two vertices are adjacent if and only if they are adjacent in $G$.

A {\em{graph parameter}} is a mapping that assigns each graph a nonnegative integer.
A {\em{graph class}} is simply a set of graphs, typically infinite. For a graph parameter $\pi$ and a graph class $\Cc$, we say that $\Cc$ has {\em{bounded $\pi$}} if there is a constant $c\in \N$ such that $\pi(G)\leq c$ for all $G\in \Cc$.

A {\em{coloring}} of a graph is just a mapping $\phi\colon V(G)\to C$ where $C$ is a finite set of colors. In the context of colorings, we will often call the set $C$ the {\em{palette}}. We will often assume $C=[k]$ where $k\in \N$ is the number of colors used by $\phi$, but it  will sometimes be convenient to use other palettes. Proper colorings and strong odd colorings of graphs have already been defined in \cref{sec:intro}. Recall that the {\em{chromatic number}} of a graph $G$ is the least number of colors needed for a proper coloring of $G$.

For $d\in \N$, we say that a graph $H$ is a {\em{depth-$d$ minor}} of a graph $G$ if one can map every vertex $u$ of $H$ to a connected subgraph $\eta(u)$ of $G$ of radius at most $d$ so that the subgraphs $\{\eta(u)\colon u\in V(H)\}$ are pairwise vertex-disjoint, and for every edge $uv$ of $H$ there is an edge in $G$ with one endpoint in $\eta(u)$ and the other in $\eta(v)$. We have the following parameters:
\begin{itemize}
	\item $\nabla_d(G)$ is the largest average degree among depth-$d$ minors of $G$; and
	\item $\omega_d(G)$ is the largest $t$ such that $K_t$ --- the complete graph on $t$ vertices --- is a depth-$d$ minor of $G$.
\end{itemize}
We say that a class of graphs $\Cc$ has {\em{bounded expansion}} if $\Cc$ has bounded $\nabla_d$, for each $d\in \N$. Further, we say that $\Cc$ is {\em{nowhere dense}} if $\Cc$ has bounded $\omega_d$, for each $d\in \N$. Clearly, every graph class of bounded expansion is also nowhere dense. It is well-known that every proper minor-closed class has bounded expansion, and every class of bounded expansion has bounded chromatic number. A more thorough introduction to the theory of Sparsity, centered around the notions of nowhere denseness and bounded expansion, can be found in~\cite{sparsity,sparsityNotes}.

\paragraph*{Set systems and their parameters.} A {\em{set system}} is a pair $\Ss=(U,\Ff)$ consisting of a set $U$, called the {\em{universe}}, and a family $\Ff$ of subsets of $U$. The {\em{Gaifman graph}} of $\Ss$, denoted $\Gaif(\Ss)$, is the graph on the vertex set $U$ where elements $x,y\in U$ are adjacent if and only if there exists $F\in \Ff$ such that $x,y\in F$. In other words, every set from $\Ff$ induces a clique in $\Gaif(\Ss)$.

Let $\Ss=(U,\Ff)$ be a set system.
For a subset of the universe $W\subseteq U$ and a subfamily $\Qq\subseteq \Ff$, we define the {\em{induced subsystem}} $\Ss[W,\Qq]$ as follows:
\[\Ss[W,\Qq]\coloneqq (W,\{F\cap W\colon F\in \Qq\}).\]
In other words, the universe of $\Ss[W,\Qq]$ is $W$, and we keep only the sets belonging $\Qq$, each trimmed to~$W$.

Similarly as for graphs, a {\em{class of set systems}} is simply a set of set systems. Parameters of set systems and their boundedness on classes are defined analogously to graphs. We call a class of set systems $\Cc$ {\em{hereditary}} if for every $\Ss=(U,\Ff)\in \Cc$, $W\subseteq U$, and $\Qq\subseteq \Ff$, we also have $\Ss[W,\Qq]\in \Cc$. The {\em{hereditary closure}} of a class $\Cc$ consists of all set systems $\Ss[W,\Qq]$ for $\Ss=(U,\Ff)\in \Cc$, $W\subseteq U$, and $\Qq\subseteq \Ff$.

A {\em{coloring}} of a set system $\Ss=(U,\Ff)$ is just a mapping $\phi\colon U\to C$ for some palette of colors $C$. We say that $\phi$ is {\em{strong odd}} if for each $F\in \Ff$ and every color $i\in C$, the set $F\cap \phi^{-1}(i)$ is either empty or of odd cardinality.

We will exploit several parameters of set systems. In this section we discuss four of them: the semi-ladder index, the comatching index, the ladder index, and the 2VC dimension. All of them are defined as the maximum size of a certain obstruction that can be found in a set system.

\begin{definition}
	Let $\Ss=(U,\Ff)$ be a set system, and consider a sequence of elements $u_1,u_2,\ldots,u_\ell\in U$ and a sequence of sets $F_1,F_2,\ldots,F_\ell\in \Ff$. We call this pair of sequences (see \cref{fig:obstructions}):
	\begin{itemize}
		\item a {\em{semi-ladder}} if we have that $u_i\notin F_i$ for each $i\in [\ell]$, and $u_i\in F_j$ for all $i,j\in [\ell]$ with $i<j$;
		\item a {\em{comatching}} if we have that $u_i\notin F_i$ for each $i\in [\ell]$, and $u_i\in F_j$ for all $i,j\in [\ell]$ with $i\neq j$; and
		\item a {\em{ladder}} if we have that $u_i\notin F_j$ for all $i,j\in [\ell]$ with $i\geq j$, and $u_i\in F_j$ for all $i,j\in [\ell]$ with $i<j$.
	\end{itemize}
	The {\em{order}} of a semi-ladder is the length $\ell$ of the sequences forming a semi-ladder, and
	the {\em{semi-ladder index}} of $\Ss$ is the largest order of a semi-ladder present in $\Ss$. Same for comatchings and ladders.
\end{definition}

Note that every comatching is a semi-ladder and also every ladder is a semi-ladder. 
As observed in~\cite{FabianskiPST19}, Ramsey's Theorem implies that every semi-ladder of a large order contains a comatching of a large order or a ladder of a large order. In particular, we have the following: a class of set systems has bounded semi-ladder index if and only if it has bounded comatching index and bounded ladder index.

We proceed to the next parameter.

\begin{definition}
	Let $\Ss=(U,\Ff)$ be a set system. We say that a set $X\subseteq U$ is {\em{$2$-shattered}} in $\Ss$ if for every pair of distinct elements $x,y\in X$ there exists $F\in \Ff$ such that $F\cap X=\{x,y\}$. The {\em{2VC dimension}} of $\Ss$ is the largest cardinality of a subset of $U$ that is $2$-shattered in $\Ss$.
\end{definition}

The key fact about set systems of bounded 2VC dimension that we will exploit is the result of Ding, Seymour, and Winkler~\cite{DingSW94} about the connection between the packing number and the hitting number in set systems whose duals have bounded 2VC dimension. The following statement is the reformulation of this result through duality of set systems.

\begin{theorem}[{\cite[(1.1)]{DingSW94}}, reformulated through duality]\label{thm:ding}
	There exists a function $f\colon \N\times \N\to \N$ such that the following holds. Suppose $\Ss=(U,\Ff)$ is a set system of 2VC dimension at most $h$ such that $\bigcup \Ff=U$ and $\Gaif(\Ss)$ does not contain any independent set of size larger than $k$. Then there exists a subfamily $\Qq\subseteq \Ff$ such that $\bigcup \Qq=U$ and $|\Qq|\leq f(h,k)$.
\end{theorem}

We note that all the introduced parameters --- the semi-ladder/comatching/ladder indices and the 2VC dimension --- are closed under taking induced subsystems: if $\Ss=(U,\Ff)$ is a set system and $W\subseteq U$ and $\Qq\subseteq \Ff$, then the value of any of these parameters for $\Ss[W,\Qq]$ is not greater than for $\Ss$.

\input{fig-obstructions}

\paragraph*{Set systems of balls.} For a graph $G$, vertex $u$, and $d\in \N$, the {\em{radius-$d$ ball}} around $u$ is defined as
\[\Ball^G_d(u)\coloneqq \{v\in V(G)~|~\dist_G(u,v)\leq d\},\]
where $\dist_G(\cdot,\cdot)$ is the distance metric in $G$. We define the set system of radius-$d$ balls in $G$ as follows:
\[\Balls_d(G)\coloneqq (V(G),\{\Ball^G_d(u)\colon u\in V(G)\}).\]
Naturally, for a graph class $\Cc$, we may define the corresponding classes of set systems:
\[\Balls_d(\Cc)\coloneqq \{\Balls_d(G)\colon G\in \Cc\}.\]
It turns out that provided  $d$ is fixed and $\Cc$ has bounded expansion, the class of set systems $\Balls_d(\Cc)$ has bounded semi-ladder index and bounded 2VC dimension. This was proved even in a larger generality of nowhere dense classes, which subsume classes of bounded expansion. The result for the semi-ladder index is due to Fabia\'nski et al.~\cite{FabianskiPST19}, and for the 2VC dimension is due to Pilipczuk and Siebertz~\cite{PilipczukS21}


\begin{lemma}\label{lem:sl-bounded}
	For every $d\in \N$ and graph class of bounded expansion $\Cc$, the class of set systems $\Balls_d(\Cc)$ has bounded semi-ladder index.
\end{lemma}
\begin{proof}
	As proved in~\cite[Lemma~14]{FabianskiPST19}, the class $\Balls_d(\Cc)$ has bounded semi-ladder index, provided $\Cc$ is {\em{uniformly quasi-wide}}. And it is known that classes that are uniformly quasi-wide exactly coincide with nowhere dense classes~\cite{NPoM-nd}, while every class of bounded expansion is nowhere dense.
\end{proof}


\begin{lemma}\label{lem:2VC-bounded}
	For every $d\in \N$ and graph class of bounded expansion $\Cc$, the class of set systems $\Balls_d(\Cc)$ has bounded 2VC-dimension.
\end{lemma}
\begin{proof}
	As proved in \cite[Corollary~11]{PilipczukS21}, if a graph $G$ does not contain $K_t$ as a depth-$d$ minor, then the 2VC dimension of the set system $\Balls_d(G)$ is bounded by $t-1$. Since $\Cc$ has bounded expansion, it is in particular nowhere dense, hence there exists a $t\in \N$ such that no member of $\Cc$ contains $K_t$ as a depth-$d$ minor. So the 2VC dimension of $\Balls_d(\Cc)$ is bounded by~$t-1$.
\end{proof}

%% file: fig-obstructions.tex
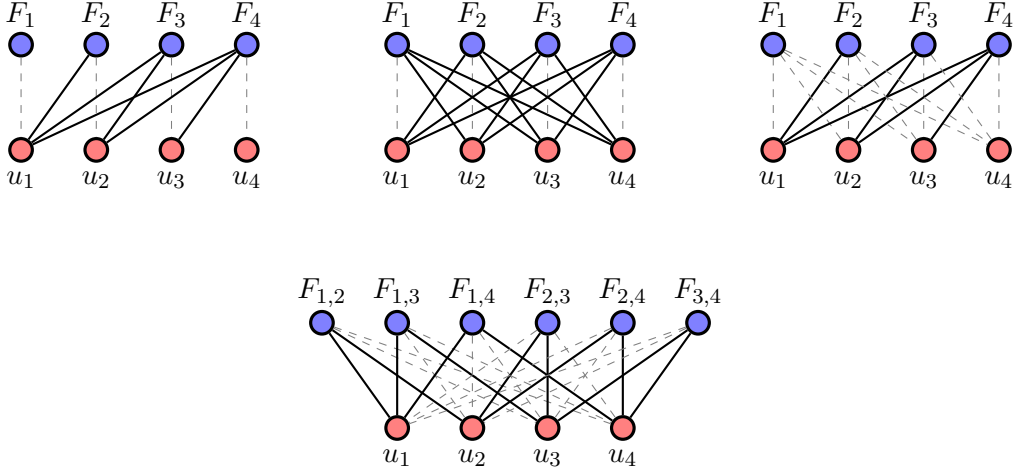
\begin{figure}
	\centering
\begin{tikzpicture}
	\tikzstyle{vertex}=[circle,draw=black,fill=yellow,minimum size=0.3cm,inner sep=0pt]
	\tikzstyle{rvertex}=[circle,very thick,draw=black,fill=red!50,minimum size=0.3cm,inner sep=1pt]
	\tikzstyle{bvertex}=[circle,very thick,draw=black,fill=blue!50,minimum size=0.3cm,inner sep=1pt]
	
	\begin{scope}[shift={(-5,5)}]
	
	\foreach \i in {1,2,3,4} {
		\node[bvertex] (f\i) at (-2.5+\i,0.7) {};
		\node at (-2.5+\i,1.1) {$F_\i$};
		\node[rvertex] (u\i) at (-2.5+\i,-0.7) {};
		\node at (-2.5+\i,-1.1) {$u_\i$};
		\draw[dashed,black!50] (f\i) -- (u\i);
	}
	\foreach \i/\j in {1/2,1/3,1/4,2/3,2/4,3/4} {
		\draw[thick,black] (f\j) -- (u\i);
	}
	\end{scope}
	
	\begin{scope}[shift={(0,5)}]
		
		\foreach \i in {1,2,3,4} {
			\node[bvertex] (f\i) at (-2.5+\i,0.7) {};
			\node at (-2.5+\i,1.1) {$F_\i$};
			\node[rvertex] (u\i) at (-2.5+\i,-0.7) {};
			\node at (-2.5+\i,-1.1) {$u_\i$};
			\draw[dashed,black!50] (f\i) -- (u\i);
		}
		\foreach \i/\j in {1/2,1/3,1/4,2/3,2/4,3/4} {
			\draw[thick,black] (f\j) -- (u\i);
			\draw[thick,black] (f\i) -- (u\j);
		}
	\end{scope}

\begin{scope}[shift={(5,5)}]
	
	\foreach \i in {1,2,3,4} {
		\node[bvertex] (f\i) at (-2.5+\i,0.7) {};
		\node at (-2.5+\i,1.1) {$F_\i$};
		\node[rvertex] (u\i) at (-2.5+\i,-0.7) {};
		\node at (-2.5+\i,-1.1) {$u_\i$};
		\draw[dashed,black!50] (f\i) -- (u\i);
	}
	\foreach \i/\j in {1/2,1/3,1/4,2/3,2/4,3/4} {
		\draw[thick,black] (f\j) -- (u\i);
		\draw[dashed,black!50] (f\i) -- (u\j);
	}
\end{scope}
	
	\begin{scope}[shift={(0,1.3)}]

		\foreach \i in {1,2,3,4} {
			\node[rvertex] (u\i) at (-2.5+\i,-0.7) {};
			\node at (-2.5+\i,-1.1) {$u_\i$};
		}
		\foreach \a/\i/\j in {1/1/2,2/1/3,3/1/4,4/2/3,5/2/4,6/3/4} {
			\node[bvertex] (f\i\j) at (-3.5+\a,0.7) {};
			\node at (-3.5+\a,1.1) {$F_{\i,\j}$};
		}
		\foreach \i/\j/\a/\b in {1/2/3/4,1/3/2/4,1/4/2/3,2/3/1/4,2/4/1/3,3/4/1/2} {
			\draw[thick,black] (f\i\j) -- (u\i);
			\draw[thick,black] (f\i\j) -- (u\j);
			\draw[dashed,black!50] (f\i\j) -- (u\a);
			\draw[dashed,black!50] (f\i\j) -- (u\b);
		}
	\end{scope}
	
\end{tikzpicture}
\caption{A semi-ladder of order $4$, a comatching of order $4$, a ladder of order $4$, and a $2$-shattered set of size $4$. Solid edges represent membership, dashed gray edges represent non-membership, lack of an edge represents no specification.}\label{fig:obstructions}
\end{figure}

%% file: subchromatic.tex
\section{Subchromatic number}\label{sec:subchromatic}

In this section we discuss another parameter of graphs and set systems: the subchromatic number. Contrary to the parameters discussed in \cref{sec:prelims}, the definition is based on the existence of certain~colorings.

\begin{definition}
	A {\em{cluster graph}} is a graph whose every connected component is a clique.
	For a graph $G$, a {\em{subcoloring}} of $G$ is a function $\lambda\colon V(G)\to C$, for some palette of colors $C$, satisfying the following property: for every color $i\in C$, the induced subgraph $G[\lambda^{-1}(i)]$ is a cluster graph. The {\em{subchromatic number}} of a graph $G$ is the minimum number of colors needed for a subcoloring of $G$. The {\em{subchromatic number}} of a set system $\Ss$ is the subchromatic number of its Gaifman graph $\Gaif(\Ss)$.
\end{definition}

For a graph $G$ and $r\in \N$, the {\em{$r$th power}} of $G$, denoted $G^r$, is the graph with the same vertex set as $G$ and adjacency defined as follows: any two vertices $u,v$ are adjacent in $G^r$ if and only if $\dist(u,v)\leq r$. Naturally, for a graph class $\Cc$, we denote $\Cc^r\coloneqq \{G^r\colon G\in \Cc\}$. As proved by Ne\v{s}et\v{r}il et al.~\cite{NesetrilMPZ20}, powers of bounded expansion classes have bounded subchromatic number, in the following sense.

\begin{theorem}[{\cite[Corollary~11]{NesetrilMPZ20}}]\label{thm:subchromatic-simple}
	For every $r\in \N$ and graph class of bounded expansion $\Cc$, the graph class $\Cc^r$ has bounded subchromatic number.
\end{theorem}

It can be easily seen that for any graph $G$ and $d\in \N$, the graph $\Gaif(\Balls_d(G))$ is equal to $G^{2d}$. Therefore, \cref{thm:subchromatic-simple} implies that whenever $\Cc$ is a graph class of bounded expansion and $d\in \N$ is fixed, the class of set systems $\Balls_d(\Cc)$ has bounded subchromatic number. Unfortunately, for the proof of \cref{thm:main} we will need a stronger property --- that this statement also hold for the hereditary closure of $\Balls_d(\Cc)$. (Note here that the subchromatic number of a set system is {\em{not}} closed under taking induced subsystems.) More precisely, we will need the following.

\begin{theorem}\label{thm:subchromatic}
	For every $d\in \N$ and graph class of bounded expansion $\Cc$, the hereditary closure of the class of set systems $\Balls_d(\Cc)$ has bounded subchromatic number.
\end{theorem}

We were not able to derive \cref{thm:subchromatic} directly from \cref{thm:subchromatic-simple}, so we resorted to proposing a completely new proof, based on a different proof method and relying on much stronger tools. In particular, we will use techniques related to logic and first-order transductions.

The remainder of this section is devoted to the proof of \cref{thm:subchromatic}.

\paragraph*{Transductions.} We assume reader's familiarity with the basic terminology of first-order logic on graphs and first-order transductions; see~\cite{PilipczukSurvey} for a broad introduction. Let us only recall that a (first-order) transduction $\Tf$ consists of a finite set of unary predicates $C$ and a first-order formula $\varphi(x,y)$ working over $C$-colored graphs, where a {\em{$C$-colored graph}} is a graph $G$ expanded with an interpretation of each predicate from $C$ as a subset of vertices of $G$. We assume $\varphi(x,y)$ to be symmetric: $G\models \varphi(u,v)\Leftrightarrow \varphi(v,u)$ for every $C$-colored graph $G$ and vertices $u,v$. Now, we say that $H$ is {\em{transducible}} from $G$ using $\Tf$ if $H$ can be obtained from $G$ by means of the following procedure:
\begin{itemize}
	\item Arbitrarily choose the interpretation of the predicates from $C$ as subsets of vertices of $G$, thus obtaining a $C$-colored graph $G^+$.
	\item Obtain a graph $H^+$ from $G^+$ by interpreting the new adjacency relation using $\varphi$: $V(H^+)=V(G^+)$, and two vertices $u,v$ are adjacent in $H^+$ if and only if $\varphi(u,v)$ holds in $G^+$.
	\item Drop all the predicates from $C$ and output an arbitrary induced subgraph of $H^+$ as $H$.
\end{itemize}
For a graph $G$, by $\Tf(G)$ we denote the set of all graphs $H$ that are transducible from $G$ using $\Tf$. And for a graph class $\Cc$, we define $\Tf(\Cc)\coloneqq \bigcup_{G\in \Cc} \Tf(G)$. For two graph classes $\Cc,\Dd$, we say that $\Dd$ is {\em{transducible}} from $\Cc$ if $\Dd\subseteq \Tf(\Cc)$ for some transduction $\Tf$. Since transductions are compositional --- the composition of two transductions is again a transduction --- the relation of transducibility is transitive: If $\Dd$ is transducible from $\Cc$, and $\Ee$ is transducible from $\Dd$, then $\Ee$ is transducible from $\Cc$.

We call a graph class $\Cc$ {\em{monadically dependent}} if the class of all graphs is not transducible from $\Cc$. Note that by transitivity, if $\Dd$ is transducible from $\Cc$ and $\Dd$ is not monadically dependent, then $\Cc$ is not monadically dependent either.
As proved by Adler and Adler~\cite{adler2014interpreting}, every nowhere dense class of graphs is monadically dependent. This in particular applies to every class of bounded expansion.

\paragraph*{Proof plan.}
We say that a graph class $\Dd$ has {\em{structurally bounded expansion}} if $\Dd$ is transducible from some class $\Cc$ of bounded expansion. The main technical hurdle will be in proving the following statement.

\begin{restatable}{theorem}{sbe}\label{thm:sbe-subchromatic}
	Let $t\in \N$ and $\Dd$ be a graph class of structurally bounded expansion such that no member of $\Dd$ contains $K_{t,t}$ as an induced subgraph. Then $\Dd$ has bounded subchromatic number. 
\end{restatable}

Before we discuss the approach to \cref{thm:sbe-subchromatic}, let us derive \cref{thm:subchromatic} from it.

\begin{proof}[Proof of \cref{thm:subchromatic} using \cref{thm:sbe-subchromatic}]
	Let $\Ee$ be the hereditary closure of $\Balls_d(\Cc)$ and $\Dd\coloneqq \Gaif(\Ee)$, that is, $\Dd$ is the class of Gaifman graphs of the set systems from $\Ee$. Our goal is to prove that $\Dd$ has bounded subchromatic number. 
	
	Unraveling the definitions, we have the following: for every graph $H\in \Dd$, there exists a graph $G\in \Cc$ and two subsets of vertices $A,B\subseteq V(G)$ such that $V(H)=A$ and any two distinct vertices $u,v\in A$ are adjacent in $H$ if and only if there exists $w\in B$ such that $\dist_G(u,w)\leq d$ and $\dist_G(v,w)\leq d$.
	
	This characterization is clearly expressible as a first-order transduction, which marks the set $B$ using a unary predicate, interprets the new edge relation using the formula $\varphi(x,y)=\exists z\ B(z)\wedge \delta_d(x,z)\wedge \delta_d(y,z)$, where $\delta_d(x,y)$ is the first-order formula verifying that the distance between $x$ and $y$ is at most $d$, and restricts the vertex set to $A$. Thus, we have the following.
	
	\begin{claim}\label{cl:trans}
		$\Dd$ is transducible from $\Cc$.
	\end{claim}

	Further, we observe that graphs from $\Dd$ exclude some complete bipartite as an induced subgraph.
	
	\begin{claim}\label{cl:biclique-free}
		There exists $t\in \N$ such that no graph from $\Dd$ contains $K_{t,t}$ as an induced subgraph
	\end{claim}
	\begin{claimproof}
		Suppose, for contradiction, that for every $t\in \N$ there exists a graph in $\Dd$ that contains $K_{t,t}$ as an induced subgraph. By the characterization of the graphs from $\Dd$ explained above, this means that for every $t\in \N$, there is a graph $G\in \Cc$, vertex subsets $A,B\subseteq V(G)$, and distinct vertices $u_1,\ldots,u_t,v_1,\ldots,v_t\in A$, such that the following two conditions are satisfied:
		\begin{itemize}
			\item For every pair $(i,j)\in [t]\times[t]$, there exists a vertex $w_{i,j}\in B$ such that $\dist_G(u_i,w_{i,j})\leq d$ and $\dist_G(v_j,w_{i,j})\leq d$.
			\item For every pair $(i,j)\in [t]\times[t]$ with $i\neq j$, there is no vertex $w\in B$ that would satisfy $\dist_G(u_i,w_{i,j})\leq d$ and $\dist_G(u_j,w_{i,j})\leq d$, and there is no vertex $w\in B$ that would satisfy $\dist_G(v_i,w_{i,j})\leq d$ and $\dist_G(v_j,w_{i,j})\leq d$.
		\end{itemize}
		Note that the second point above implies that for each $(i,j)\in [t]\times [t]$, the only two vertices among $\{u_1,\ldots,u_t,v_1,\ldots,v_t\}$ that are at distance at most $d$ from $w_{i,j}$ are $u_i$ and $v_j$. In particular, the vertices $w_{i,j}$ for $(i,j)\in [t]\times [t]$ are pairwise different.
		
		Consider now a transduction that marks the set $\{u_i,v_i\colon i\in [t]\}$ using a unary predicate $P$, interprets a new edge relation using the formula $\varphi(x,y)=\exists z \ P(z)\wedge \delta_d(x,z)\wedge \delta_d(y,z)$, and outputs the subgraph induced by the vertices $\{w_{i,j}\colon (i,j)\in [t]\times [t]\}$. Note that in the output graph, vertices $w_{i,j}$ and $w_{i',j'}$ are adjacent if and only if $i=i'$ or $j=j'$, hence the output graph is the so-called $t\times t$ {\em{rook graph}}. Thus, applying the transduction described above to $\Cc$ allows us to conclude that the class of rook graphs is transducible from $\Cc$. However, it is known that the class of rook graphs is not monadically dependent, see e.g.~\cite[Lemma~3]{PilipczukSurvey}. This implies that $\Cc$ is not monadically dependent either, a contradiction with the assumption that $\Cc$ has bounded expansion.
	\end{claimproof}
	
	Since $\Cc$ has bounded expansion, from \cref{cl:trans} we infer that $\Dd$ has structurally bounded expansion. Together with \cref{cl:biclique-free}, this ensures that the assumptions of \cref{thm:sbe-subchromatic} are satisfied. So we conclude that $\Dd$ has bounded subchromatic number.
\end{proof}

We are left with proving \cref{thm:subchromatic}. The main idea is to use the characterization of graph classes of structurally bounded expansion through {\em{low shrubdepth colorings}}, due to Gajarsk\'y et al.~\cite{GajarskyKNMPST20}. This allows us to split the proof of \cref{thm:subchromatic} into a sequence of three steps, each achieving a greater generality. First we treat cographs of bounded depth, then classes of bounded shrubdepth, and finally classes of structurally bounded expansion.

\paragraph*{Cographs of bounded depth.} Recall that a {\em{cotree}} of a graph $G$ is a rooted tree $T$ satisfying the following properties:
\begin{itemize}
	\item The leaf set of $T$ coincides with the vertex set of $G$.
	\item The non-leaf nodes of $T$ are divided into two types: the {\em{join nodes}} and the {\em{union nodes}}. We require that $T$ encodes the adjacency relation of $G$ in the following way: For any two distinct vertices $u,v$ of $G$, $u$ and $v$ are adjacent in $G$ if and only if the lowest common ancestor in $T$ is a join node. 
\end{itemize}
In other words, if $x$ is a join node in $T$, then leaves from different subtrees rooted at the children of $x$ are adjacent; and if $x$ is a union node, then they are non-adjacent. The {\em{depth}} of a cotree $T$ is the maximum number of nodes on a root-to-leaf path in $T$. We say that $G$ is a {\em{cograph}} if $G$ admits a cotree, and $G$ is a cograph of {\em{depth}} at most $d$ if $G$ admits a cotree of depth at most $d$.

Clearly, cographs are closed under complementation (replacing all edges with non-edges and vice versa), because this corresponds to swapping join and union nodes in a cotree. Also, it is known that cographs are {\em{perfect}}: the chromatic number of a cograph is equal to the maximum size of a clique contained in it. From these two facts we can trivially derive the following.

\begin{observation}\label{obs:cograph-perfect}
	Let $G$ be a cograph that does not contain an independent set of size $k$. Then the vertex set of $G$ can be partitioned into at most $k-1$ cliques in $G$.
\end{observation}

Our first goal is to prove the conclusion of \cref{thm:sbe-subchromatic} for cographs of bounded depth.

\begin{lemma}\label{lem:cograph-subchromatic}
	Suppose $G$ is a cograph of depth at most $d$ such that $G$ does not contain $K_{t,t}$ as an induced subgraph, for some positive $t\in \N$. Then the subchromatic number of $G$ is at most $1+(d-1)(t-1)$.
\end{lemma}
\begin{proof}
	We proceed by induction on the depth $d$. For $d=1$ the claim is trivial, for $G$ consists of one vertex.
	Suppose then that $d\geq 2$. Let then $T$ be a cotree of $G$ of depth at most $d$. For a node $x$ of $T$, by $V_x$ we denote the set of leaves of $T$ that are descendants of $x$ (where every node is considered a descendant of itself), and we write $G_x\coloneqq G[V_x]$. Further, let $r$ be the root of $T$. We distinguish two cases, depending on the type of $r$.
	
	If $r$ is a union node, then for every child $x$ of $r$, we may find by induction a subcoloring $\lambda_x$ of $G_x$ with $1+(d-2)(t-1)$ colors. Noting that $G$ is the disjoint union of graphs $G_x$ over the children $x$ of $r$, a subcoloring of $G$ can be constructed by taking the union of colorings $\lambda_x$, using for each of them the same palette of $1+(d-2)(t-1)\leq 1+(d-1)(t-1)$ colors.
	
	Suppose now that $r$ is a join node. We note the following:
	\begin{claim}
		There is at most one child $x$ of $r$ in $T$ such that $G_x$ contains an independent set of size $t$.
	\end{claim} 
	\begin{claimproof}
		If there are two distinct children $x,x'$ such that $G_x$ contains an independent set $I$ of size $t$ and $G_{x'}$ contains an independent set $I'$ of size $t$, then the set $I\cup I'$ induces a $K_{t,t}$ in $G$; a contradiction.
	\end{claimproof}
	
	Let then $x$ be the unique child of $r$ such that $G_x$ contains an independent set of size $t$; or we set $x$ to be any child of $r$, if there is no child satisfying this property. By induction, we may fix a subcoloring $\lambda_x$ of $G_x$ using at most $1+(d-2)(t-1)$ colors. Next, by \cref{obs:cograph-perfect}, for each child $y$ different from $x$ we may partition $V_y$ into cliques $K_y^1,\ldots,K_y^{t-1}$ (some possibly empty). For $i\in [t-1]$, define \[K^i\coloneqq \bigcup \{K^i_y\colon \textrm{$y$ is a child of $r$ different from $x$}\}.\]
	Note that each set $K^i$ is a clique in $G$, for $r$ is a join node. Therefore, the mapping $\lambda'\colon (V(G)\setminus V_x)\to [t-1]$ defined by assigning color $i$ to every vertex of $K^i$ is a subcoloring of $G-V_x$. We may now obtain a subcoloring of $G$ with at most $1+(d-2)(t-1)+(t-1)=1+(d-1)(t-1)$ colors by taking the union of subcolorings $\lambda_x$ and $\lambda'$, using disjoint palettes for them.
\end{proof}

\paragraph*{Classes of bounded shrubdepth.} Next, we treat classes of bounded shrubdepth. Shrubdepth is a graph parameter introduced by Ganian et al.~\cite{GanianHNOM19}. The underlying decomposition notion is called {\em{connection model}}, and intuitively, it generalizes the concept of a cograph of bounded depth by allowing a more elaborate mechanism for encoding the adjacency relation in a bounded-depth tree.

\begin{definition}
	Let $G$ be a graph. A {\em{connection model}} for $G$ consists of a finite set of labels $\Lambda$, a labelling $\lambda\colon V(G)\to \Lambda$, a rooted tree $T$ whose leaf set coincides with the vertex set of $G$, and, for every non-leaf node $x$ of $T$, a symmetric relation $M_x\subseteq \Lambda\times \Lambda$. We require that for any two distinct vertices $u,v$, we have
	\[uv\textrm{ are adjacent in }G\qquad\textrm{if and only if}\qquad (\lambda(u),\lambda(v))\in M_x,\]
	where $x$ is the lowest common ancestor of $u$ and $v$ in $T$. We say that a class of graphs $\Dd$ has {\em{bounded shrubdepth}} if there exists $d,m\in \N$ such that every member of $\Dd$ admits a connection model of depth at most $d$ and using a label set of size $m$. As with cographs, the {\em{depth}} of a connection model is the maximum number of nodes on a root-to-leaf path in $T$.
\end{definition}

Note the following.

\begin{observation}\label{obs:shb-cograph}
	Let $(\Lambda,\lambda,T,\{M_x\}_{x\in V(T)\setminus V(G)})$ be a connection model of depth at most $d$ of a graph~$G$. Then for every $i\in \Lambda$, the graph $G[\lambda^{-1}(i)]$ is a cograph of depth at most $d$.
\end{observation}
\begin{proof}
	A cotree of depth at $d$ of $G[\lambda^{-1}(i)]$ can be constructed by taking $T$, removing all the nodes which do not have a descendant leaf with label $i$, and making every remaining non-leaf node $x$ a join node if $(i,i)\in M_x$, and a union node if $(i,i)\notin M_x$.
\end{proof}

This allows us to easily lift the conclusion of \cref{lem:cograph-subchromatic} to the setting of classes of bounded shrubdepth.

\begin{lemma}\label{lem:shrub-subchromatic}
	Let $t\in \N$ and $\Dd$ be a graph class of  bounded shrubdepth such that no member of $\Dd$ contains $K_{t,t}$ as an induced subgraph. Then $\Dd$ has bounded subchromatic number.
\end{lemma}
\begin{proof}
	By definition, there are $d,m\in \Dd$ such that every member of $\Dd$ admits a connection model of depth at most $d$ and using $m$ labels. By \cref{obs:shb-cograph}, we conclude that for every $G\in \Dd$, the vertex set of $G$ can be partitioned into at most $m$ subsets, each inducing a cograph of depth at most $d$. We may now apply \cref{lem:cograph-subchromatic} to obtain a subcoloring of each of these cographs. Assuming the obtained subcolorings use disjoint palettes, we may construct a subcoloring of $G$ by taking their union. This subcoloring uses at most $m\cdot (1+(d-1)(t-1))$ colors, which is a constant depending only on $\Dd$ and $t$.
\end{proof}

\paragraph*{Classes of structurally bounded expansion.}

We will use the following result of Gajarsk\'y et al.~\cite{GajarskyKNMPST20} that classes of structurally bounded expansion admit {\em{low shrubdepth colorings}}.

\begin{theorem}[\cite{GajarskyKNMPST20}]\label{thm:sbe}
	Let $\Dd$ be a graph class of structurally bounded expansion. Then for every $p\in \N$, there is a finite palette of colors $C_p$ and a graph class $\Ee_p$ of bounded shrubdepth such that the following holds: For every graph $G\in \Cc$, there is a coloring $\alpha\colon V(G)\to C_p$ such that for every subset of colors $I\subseteq C_p$ with $|I|\leq p$, we have $G[\alpha^{-1}(I)]\in \Ee_p$.
\end{theorem}

Now, \cref{thm:sbe} allows us to easily lift the conclusion of \cref{lem:shrub-subchromatic} to the setting of classes of structurally bounded expansion. Thus we finish the proof of \cref{thm:sbe-subchromatic}, restated below for convenience.

\sbe*
\begin{proof}
	Consider any $G\in \Dd$.
	By \cref{thm:sbe} applied to $\Dd$ and $p=1$, we find a coloring $\alpha\colon V(G)\to C$ such that each color in $\alpha$ induces a graph belonging to $\Ee$, where $C\coloneqq C_1$ and $\Ee\coloneqq \Ee_1$ is the finite palette and the class of bounded shrubdepth provided by \cref{thm:sbe}. Further, by \cref{lem:shrub-subchromatic}, for each color $i\in C$ we may find a subcoloring $\lambda_i$ of $G[\alpha^{-1}(i)]$ using at most $q$ colors, where $q$ is a constant depending only on $\Ee$ and $t$. Assuming that each coloring $\lambda_i$ uses a disjoint palette of $q$ colors, we find that $\bigcup_{i\in C} \lambda_i$ is a subcoloring of $G$ using at most $|C|\cdot q$ colors. This proves that the subchromatic number of every graph in $\Dd$ is bounded by $|C|\cdot q$, which is a constant depending only on $\Dd$ and $t$.
\end{proof}

%% file: proof.tex
\section{Main proof}\label{sec:proof}

In this section, we prove \cref{thm:main-set-system} and then derive \cref{thm:main} from it. The idea for the proof of \cref{thm:main-set-system} is to apply induction on the semi-ladder index. The main push is encapsulated in the following technical lemma.

\begin{lemma}\label{lem:step}
	Suppose $\Ss=(U,\Ff)$ is a set system satisfying the following conditions:
	\begin{itemize}
		\item The 2VC dimension of $\Ss$ is bounded by a constant $h\in \N$.
		\item The subchromatic number of $\Ss$ is bounded by a constant $s\in \N$.
		\item We have $U\notin \Ff$; that is, $\Ff$ does not contain the set consisting of all the elements of the universe.
	\end{itemize}
	Then there is a constant $k\in \N$, depending only on $h$ and $s$, and a subcoloring $\lambda$ of $\Gaif(\Ss)$ with at most $k$ colors so that the following condition is satisfied: for every color $i$ and connected component $X$ of $\Gaif(\Ss)[\lambda^{-1}(i)]$, the semi-ladder index of the set system $\Ss[X,\Ff]$ is strictly smaller than that of $\Ss$.
\end{lemma}
\begin{proof}
	Denote $G\coloneqq \Gaif(\Ss)$. Since $\Ss$ has subchromatic number at most $s$, we may find a subcoloring $\kappa\colon U\to [s]$ of $G$. Recall that for each $i\in [s]$, the graph $G[\kappa^{-1}(i)]$ is a cluster graph, that is, a disjoint union of cliques. We will call those cliques {\em{clusters}} of colors $i$; and a {\em{cluster}} is a cluster of color $i$, for some $i\in [s]$. Thus, the vertex set of $G$ is partitioned into clusters, each being a clique in $G$, and distinct clusters of the same color are not adjacent. We let $\Yy$ be the set of clusters. 
	
	Consider any cluster $Y\in \Yy$. Note that $\Gaif(\Ss[Y,\Ff])$ is the complete graph on vertex set $Y$, hence it does not contain any independent set of size larger than $1$. Further, as $\Ss[Y,\Ff]$ is an induced subsystem of $\Ss$, the 2VC dimension of $\Ss[Y,\Ff]$ is at most $h$. Therefore, from \cref{thm:ding} we infer that there exists $\Qq_Y\subseteq \Ff$ such that $\bigcup \Qq_Y\supseteq Y$ and $|\Qq_Y|\leq f(h,1)$, where $f$ is the function provided by \cref{thm:ding}. Therefore, denoting $m\coloneqq f(h,1)$, we may partition $Y$ into subsets $Y^1,\ldots,Y^{m_Y}$ for some $m_Y\leq m$ so that each subset $Y^j$, for $j\in [m_Y]$, is contained in some $F\in \Ff$.
	
	The following claim is the key step of the whole proof.
	
	\begin{claim}\label{cl:semi-ladder-drop}
		For each $Y\in \Yy$ and $j\in [m_Y]$, the semi-ladder index of the set system $\Ss[Y^j,\Ff]$ is strictly smaller than that of~$\Ss$.
	\end{claim}
	\begin{claimproof}
		As we argued, there exists $F\in \Ff$ such that $Y^j\subseteq F$. Since $U\notin \Ff$, we have $F\neq U$ and therefore there exists $u\in U$ such that $u\notin F$. Observe now that every semi-ladder in $\Ss[Y^j,\Ff]$, say of order $\ell$, can be extended by the pair $(u,F)$ to a semi-ladder in $\Ss$ of order $\ell+1$. It follows that the semi-ladder index of $\Ss[Y^j,\Ff]$ is strictly smaller than the semi-ladder index of $\Ss$.
	\end{claimproof}
	
	Define now a mapping $\lambda\colon U\to [s]\times [m]$ by setting 
	\[\lambda(u)\coloneqq (\kappa(u),j),\]
	where $j\in [m]$ is such that $u\in Y^j$ for $Y\in \Yy$ that contains $u$. By construction, the connected components of the subgraphs induced by the colors of $\lambda$ are exactly the sets $Y^j$, for $Y\in \Yy$ and $j\leq m_Y$. Clearly, each of them is a clique in $G$, hence $\lambda$ is a subcoloring of $G$ with $k\coloneqq sm$ colors. And \cref{cl:semi-ladder-drop} implies that each such component induces a subsystem of strictly smaller semi-ladder index, as required.
\end{proof}

We are ready to complete the proof of \cref{thm:main-set-system}. We recall it for convenience.

\setSystem*

\begin{proof}
	We fix $h$ and $s$ for the remainder of the proof and proceed by induction on the semi-ladder index~$\ell$. Let $\Ss=(U,\Ff)$. Our goal is to prove a slightly stronger statement: we will construct a strong odd coloring $\sigma\colon U\to [f(\ell,h,s)]$ such that also in the whole universe $U$, every color appears zero times or an odd number of times. A coloring having this property will be called a {\em{totally strong odd coloring}}. The number of colors $f(\ell,h,s)$ will be governed by the following recursion:
	\begin{eqnarray*}
		f(0,h,s) & \coloneqq & 2,\\
		f(\ell,h,s) & \coloneqq & k\cdot f(\ell-1,h,s)\cdot 2^{f(\ell-1,h,s)+1}\qquad\textrm{for $\ell\geq 1$,}
	\end{eqnarray*}
	where $k$ is the constant provided by \cref{lem:step} for $h$ and $s$. For convenience, denote $p_\ell\coloneqq f(\ell,h,s)$.
	
	Consider first the base case $\ell=0$. That $\Ss$ has semi-ladder index $0$ means that $\Ff=\emptyset$ or $\Ff=\{U\}$. In any case, a totally strong odd coloring of $\Ss$ with two colors can be obtained by either assigning all the elements of $U$ the same color, provided $|U|$ is odd, or arbitrarily splitting $U$ into two colors of odd cardinality, provided $|U|$ is even.  
	
	We proceed to the induction step. First, we may assume that $U\notin \Ff$. Indeed, since we anyway aim at constructing a totally strong odd coloring, constructing it for the subsystem obtained by removing $U$ from $\Ff$ will automatically satisfy also the requirement imposed by $U$. Note here that removing $U$ from $\Ff$ means passing to an induced subsystem, in which all the relevant parameters --- the semi-ladder index, the 2VC dimension, and the maximum subchromatic number among induced subsystems --- are still bounded by $\ell,h,s$, respectively.
	
	Having assumed that $U\notin \Ff$, we may apply \cref{lem:step} to construct a subcoloring $\lambda\colon U\to [k]$ of $G\coloneqq \Gaif(\Ss)$ with the following property: for every color $i\in [k]$ and a connected component $X$ of $G[\lambda^{-1}(i)]$, the induced subsystem $\Ss[X,\Ff]$ has semi-ladder index at most $\ell-1$. Similarly to the proof of \cref{lem:step}, such connected components will be called {\em{clusters}}, and a cluster of color $i$ is one whose elements are colored $i$ in $\lambda$. The set of all clusters will be denoted $\Xx$; this is a partition of $U$. 
	
	By induction, for every cluster $X\in \Xx$ we may find a totally strong odd coloring $\sigma_X\colon X\to [p_{\ell-1}]$ of $\Ss[X,\Ff]$. We now classify clusters $X$ according to their color in $\lambda$ and the set of colors present within them in $\sigma_X$. That is, for $i\in [k]$ and $T\subseteq [p_{\ell-1}]$, we let $\Xx^{i,T}$ be the set of all clusters $X\in \Xx$ such that
	\begin{itemize}
		\item $X$ is of color $i$; and
		\item $\sigma_X(X)=T$.
	\end{itemize}
	Consider any $X\in \Xx^{i,T}$ and
	note that since $\sigma_X$ is totally strong odd, for each color $j\in [p_{\ell-1}]$, the number of elements of $X$ colored $j$ under $\sigma_X$ is zero if $j\notin T$, and odd if $j\in T$.
	
	Now, for each $i\in [k]$ and $T\subseteq [p_{\ell-1}]$, define a partition $\{\Xx^{i,T}_1,\Xx^{i,T}_2\}$ of $\Xx^{i,T}$ as follows:
	\begin{itemize}
		\item If $|\Xx^{i,T}|$ is odd, then $\Xx^{i,T}_1\coloneqq \Xx^{i,T}$ and $\Xx^{i,T}_2\coloneqq \emptyset$.
		\item Otherwise, we let $\{\Xx^{i,T}_1,\Xx^{i,T}_2\}$ be an arbitrary partition of $\Xx^{i,T}$ into two subsets of odd size.
	\end{itemize}
	Denote $P\coloneqq [k]\times 2^{[p_{\ell-1}]}\times [2]$ for convenience.

	With all the necessary objects defined, we may construct the sought coloring $\sigma\colon U\to [p_{\ell-1}]\times P$ as follows: for every $u\in U$, say belonging to a cluster $X\in \Xx^{i,T}_a$ for some $(i,T,a)\in P$, we set
	\[\sigma(u)\coloneqq (\sigma_X(u),i,T,a).\]
	Clearly, $\sigma$ uses at most $k\cdot p_{\ell-1}\cdot 2^{p_{\ell-1}+1}=p_\ell$ colors. So it remains to verify that $\sigma$ is a totally strong odd coloring of $\Ss$.
	
	Consider any $F\in \Ff$. Since the elements of the sets $\bigcup \Xx^{i,X}_a$ for different tuples $(i,T,a)\in P$ are colored with disjoint palettes, it suffices to show that for each tuple $(i,T,a)\in P$, every color appears in $F\cap \bigcup \Xx^{i,X}_a$ either zero times or an odd number of times. As $\bigcup \Xx^{i,X}_a\subseteq \lambda^{-1}(i)$, we have that $G[\bigcup \Xx^{i,X}_a]$ is an induced subgraph of $G[\lambda^{-1}(i)]$, and hence it is a cluster graph. Note that $F\cap \bigcup \Xx^{i,X}_a$ is a clique in $G[\bigcup \Xx^{i,X}_a]$, by the definition of the Gaifman graph. It follows that $F\cap \bigcup \Xx^{i,X}_a$ is entirely contained in a single cluster of $\Xx^{i,X}_a$, or in other words, there exists $X\in \Xx^{i,X}_a$ such that $F\cap \bigcup \Xx^{i,X}_a=F\cap X$. It now follows from $\sigma_X$ being a strong odd coloring of $\Ss[X,\Ff]$ that every color appears in $F\cap X$ either zero times or an odd number of times.
	
	We are left with checking that in total in the whole $U$, every color $(j,i,T,a)\in [p_{\ell-1}]\times P$ appears either zero times or an odd number of times. By construction, this color does not appear at all if $j\notin T$, and otherwise it appears exactly in the clusters belonging to $\Xx^{i,T}_a$, and in each of them it is featured an odd number of times. Since $|\Xx^{i,T}_a|$ is either zero or odd by construction, it follows that $(j,i,T,a)$ appears either zero times or an odd number of times in the whole $U$.
\end{proof}

We conclude by deriving \cref{thm:main} from all the tools we have established.

\main*
\begin{proof}
	Let $\Dd$ be the hereditary closure of the class of set systems $\Balls_d(\Cc)$. By \cref{lem:sl-bounded,lem:2VC-bounded,thm:subchromatic}, there are $\ell,h,s\in \N$ depending only on $d$ and $\Cc$ such that $\Dd$ has semi-ladder index at most $\ell$, 2VC dimension at most $h$, and subchromatic number at most $s$. Further, since $\Cc$ has bounded expansion, it also has bounded chromatic number, say by $b\in \N$.
	
	Consider any graph $G\in \Cc$ and let $\Ss=(V(G),\Ff)\coloneqq \Balls_d(G)$. Let $\alpha\colon V(G)\to [b]$ be a proper coloring of $G$ with $b$ colors. For each color $i\in [b]$, let $\Ss_i\coloneqq \Ss[\alpha^{-1}(i),\Ff]$. Noting that $\Ss_i\in \Dd$, we may apply to \cref{thm:main-set-system} to obtain a strong odd coloring $\sigma_i\colon \alpha^{-1}(i)\to [k]$ of $\Ss_i$, where $k\coloneqq f(\ell,h,s)$ and $f$ is the function provided by \cref{thm:main-set-system}; thus $k$ is a constant depending only on $\Cc$ and $d$. We now define a coloring $\phi\colon V(G)\to [b]\times [k]$ with $c\coloneqq bk$ colors by setting
	\[\phi(u)\coloneqq (\alpha(u),\sigma_{\alpha(u)}(u)).\]
	Clearly $\phi$ is proper, because it is a refinement of the proper coloring $\alpha$. Also, that each $\sigma_i$ is a strong odd coloring of $\Ss_i$ implies that $\phi$ is a strong odd coloring of $\Ss$. So $\phi$ satisfies all the requested properties.
\end{proof}

%% file: conclusions.tex
\section{Conclusions}

In this paper we argued that every graph class of bounded expansion has bounded strong odd chromatic number. In fact, we proved a more general statement, where the zero-or-odd condition on the number of times a color is featured is applied to balls of a fixed radius $d$, instead of distance-$1$ neighborhoods. The zero-or-odd condition appeared essentially only in the proof of \cref{thm:main-set-system}. It is not hard to see that the same proof method goes through with minimal modifications if this condition was be replaced by the following more general one: in every ball of radius $d$, every color is featured either zero times, or $1 \bmod m$ times, where $m\in \N$ is a fixed modulus. We leave the details to the reader.

The key ingredient in our proof is \cref{thm:main-set-system}, which states that bounding the three parameters --- the semi-ladder index, the 2VC dimension, and the maximum subchromatic number among induced subsystems --- suffices to derive a bound on the strong odd chromatic number of a set system. This result may serve as a motivation to seek a characterization of hereditary classes of set systems that have bounded strong odd chromatic number. It is easy to see that the boundedness of the ladder index, of the comatching index, and of the 2VC dimension, are all necessary conditions for such a characterization, because the families of induced obstructions for those parameters have unbounded strong odd chromatic number. However, it is not clear to whether bounding the subchromatic number is necessary as well, and to what~extent.

%% file: main.bbl
\begin{thebibliography}{10}

\bibitem{adler2014interpreting}
Hans Adler and Isolde Adler.
\newblock Interpreting nowhere dense graph classes as a classical notion of
  model theory.
\newblock {\em European Journal of Combinatorics}, 36:322--330, 2014.

\bibitem{CaroPST24}
Yair Caro, Mirko Petru\v{s}evski, Riste \v{S}krekovski, and Zsolt Tuza.
\newblock On strong odd colorings of graphs.
\newblock {\em ArXiv preprint}, abs/2410.02336, 2024.

\bibitem{DingSW94}
Guoli Ding, Paul~D. Seymour, and Peter Winkler.
\newblock Bounding the vertex cover number of a hypergraph.
\newblock {\em Combinatorica}, 14(1):23--34, 1994.

\bibitem{FabianskiPST19}
Grzegorz Fabia\'nski, Micha\l{} Pilipczuk, Sebastian Siebertz, and Szymon
  Toru\'nczyk.
\newblock Progressive algorithms for domination and independence.
\newblock In Rolf Niedermeier and Christophe Paul, editors, {\em 36th
  International Symposium on Theoretical Aspects of Computer Science, {STACS}
  2019}, volume 126 of {\em LIPIcs}, pages 27:1--27:16. Schloss Dagstuhl ---
  Leibniz-Zentrum f{\"{u}}r Informatik, 2019.

\bibitem{GajarskyKNMPST20}
Jakub Gajarsk{\'{y}}, Stephan Kreutzer, Jaroslav Ne\v{s}et\v{r}il, Patrice
  {Ossona de Mendez}, Micha\l{} Pilipczuk, Sebastian Siebertz, and Szymon
  Toru\'nczyk.
\newblock First-order interpretations of bounded expansion classes.
\newblock {\em {ACM} Transactions on Computational Logic}, 21(4):29:1--29:41,
  2020.

\bibitem{GanianHNOM19}
Robert Ganian, Petr Hlin\v{e}n{\'{y}}, Jaroslav Ne\v{s}et\v{r}il, Jan
  Obdr\v{z}{\'{a}}lek, and Patrice {Ossona de Mendez}.
\newblock Shrub-depth: Capturing height of dense graphs.
\newblock {\em Logical Methods in Computer Science}, 15(1), 2019.

\bibitem{GoetzeKKPPU25}
Miriam Goetze, Fabian Klute, Kolja Knauer, Irene Parada, {Juan Pablo} Pe{\~n}a,
  and Torsten Ueckerdt.
\newblock Strong odd coloring in minor-closed classes.
\newblock {\em ArXiv preprint}, abs/2505.02736, 2025.

\bibitem{Hickingbotham23}
Robert Hickingbotham.
\newblock Odd colourings, conflict-free colourings and strong colouring
  numbers.
\newblock {\em Australasian Journal of Combinatorics}, 87:160--164, 2023.

\bibitem{KwonP24}
Hyemin Kwon and Boram Park.
\newblock Strong odd coloring of sparse graphs.
\newblock {\em ArXiv preprint}, abs/2401.11653, 2024.

\bibitem{Liu24}
Chun{-}Hung Liu.
\newblock Proper conflict-free list-coloring, odd minors, subdivisions, and
  layered treewidth.
\newblock {\em Discrete Mathematics}, 347(1):113668, 2024.

\bibitem{NPoM-nd}
Jaroslav Ne\v{s}et\v{r}il and Patrice {Ossona de Mendez}.
\newblock On nowhere dense graphs.
\newblock {\em European Journal of Combinatorics}, 32(4):600--617, 2011.

\bibitem{sparsity}
Jaroslav Ne\v{s}et\v{r}il and Patrice {Ossona de Mendez}.
\newblock {\em Sparsity --- {G}raphs, {S}tructures, and {A}lgorithms},
  volume~28 of {\em Algorithms and combinatorics}.
\newblock Springer, 2012.

\bibitem{NesetrilMPZ20}
Jaroslav Ne\v{s}et\v{r}il, Patrice {Ossona de Mendez}, Micha\l{} Pilipczuk, and
  Xuding Zhu.
\newblock Clustering powers of sparse graphs.
\newblock {\em Electronic Journal of Combinatorics}, 27(4):4, 2020.

\bibitem{sparsityNotes}
Marcin Pilipczuk, Micha\l{} Pilipczuk, and Sebastian Siebertz.
\newblock Lecture notes for the course ``{S}parsity'' given at {F}aculty of
  {M}athematics, {I}nformatics, and {M}echanics of the {U}niversity of
  {W}arsaw, Winter semesters 2017/18 and 2019/20.
\newblock Available online at
  \href{https://www.mimuw.edu.pl/~mp248287/sparsity2}{https://www.mimuw.edu.pl/~mp248287/sparsity2}.

\bibitem{PilipczukSurvey}
Micha\l{} Pilipczuk.
\newblock Graph classes through the lens of logic.
\newblock {\em ArXiv preprint}, abs/2501.04166, 2025.

\bibitem{PilipczukS21}
Micha\l{} Pilipczuk and Sebastian Siebertz.
\newblock Kernelization and approximation of distance-r independent sets on
  nowhere dense graphs.
\newblock {\em European Journal on Combinatorics}, 94:103309, 2021.

\end{thebibliography}
